\documentclass{birkjour}
 \newtheorem{thm}{Theorem}[section]
 
 \newtheorem{lem}[thm]{Lemma}
 
 \theoremstyle{definition}
 
 \theoremstyle{remark}
 
 \newtheorem*{ex}{Example}
 \numberwithin{equation}{section}
 
\newcommand{\R}{{\mathbb R}}
\newcommand{\C}{{\mathbb C}}
\newcommand{\N}{{\mathbb N}}

\begin{document}

%
%
%
%
%
%
%
%
%

\title[Sturm-Liouville problems with transfer condition]
 {Sturm-Liouville problems with transfer condition Herglotz dependent on the eigenparameter -- Hilbert space formulation}

\author[C. Bartels]{Casey Bartels}

\address{%
School of Mathematics\\
University of the Witwatersrand\\
Private Bag 3, P O WITS 2050, South Africa}

\email{casey.bartels098@gmail.com}

\author[S. Currie]{Sonja Currie}
\address{%
School of Mathematics\\
University of the Witwatersrand\\
Private Bag 3, P O WITS 2050, South Africa}

\email{Sonja.Currie@wits.ac.za}

\thanks{Supported in part by the Centre for Applicable Analysis and Number Theory and  by NRF grant number IFR160209157585 Grant no. 103530.}
\author[M. Nowaczyk]{Marlena Nowaczyk}
\address{AGH University of Science and Technology\\
Faculty of Applied Mathematics\\
al. A. Mickiewicza 30, 30-059 Krakow, Poland}
\email{mno@agh.edu.pl}
\author[B. A. Watson]{Bruce Alastair Watson}
\address{%
School of Mathematics\\
University of the Witwatersrand\\
Private Bag 3, P O WITS 2050, South Africa}

\email{Bruce.Watson@wits.ac.za}

\thanks{Supported in part by the Centre for Applicable Analysis and Number Theory and  by NRF grant number IFR170214222646 Grant no. 109289.}

\subjclass{Primary 34B24; Secondary 34A36, 34B07}

\keywords{Sturm-Liouville, Transmission condition}

\date{March 26, 2018}

\begin{abstract}
We consider a Sturm-Liouville equation
$\ell y:=-y'' + qy = \lambda y$
on the intervals $(-a,0)$ and $(0,b)$ with
$a,b>0$ and $q \in L^2(-a,b)$. We impose boundary conditions 
$y(-a)\cos\alpha = y'(-a)\sin\alpha$,
  $y(b)\cos\beta = y'(b)\sin\beta$,
 where $\alpha \in [0,\pi)$ and $\beta \in (0,\pi]$, together with transmission conditions rationally-dependent on the eigenparameter via 
\begin{align*}
-y(0^+)\left(\lambda \eta -\xi-\sum\limits_{i=1}^{N} \frac{b_i^2}{\lambda -c_i}\right) &= y'(0^+) - y'(0^-),\\
y'(0^-)\left(\lambda \kappa +\zeta-\sum\limits_{j=1}^{M}\frac{a_j^2}{\lambda -d_j}\right) &= y(0^+) - y(0^-),
\end{align*}
with
 $b_i, a_j>0$ for $i=1,\dots,N,$ and $j=1,\dots,M$.
 Here we take $\eta, \kappa \ge 0$ and $N,M\in \N_0$.
The geometric multiplicity of the eigenvalues  is considered and the cases in which the multiplicity can be $2$ are characterized. An example is given to illustrate the cases. A Hilbert space formulation of the above eigenvalue problem  as a self-adjoint operator eigenvalue problem in
$L^2(-a,b)\bigoplus \C^{N^*} \bigoplus \C^{M^*}$, for suitable $N^*,M^*$, is given. The Green's function and the resolvent of the related Hilbert space operator are expressed explicitly. 
\end{abstract}

\maketitle
\section{Introduction\label{sec-intro-0}}
There has been growing interest in  spectral problems involving differential operators with discontinuity conditions. We refer to such conditions as transmission conditions (see \cite{Dehghani-Akbarfam, 
Mukhtarov-Tunc, 
Wang-Sun-Hao-Yao}), although they appear under the guise of many names. These include point interactions in the physics literature, with important examples being the $\delta$ and $\delta'$ interactions from quantum mechanics (see for example \cite{Albeverio-Gesztesy-Krohn-Holden, Eckhardt} and the references therein); interface conditions,  \cite{Zhao-Sun-Zettl}; as well as matching conditions on graphs,  \cite{Yang}. 
For an interesting exposition of transmission condition problems that arise naturally in applications we refer the reader to the book by A. N. Tikhonov and A. A. Samarskii, \cite[Chapter II]{Tikhonov-Samarskii}.

 Direct and inverse problems for continuous Sturm-Liouville equations with eigenparameter dependent boundary conditions have been studied extensively (see 
 \cite{Binding-Browne-Watson1I, Binding-Browne-Watson1, MM, Schmid-Tretter, tretter} for a sample of the literature). Investigations into Sturm-Liouville equations with discontinuity conditions depending on the spectral parameter have been thus far limited to the affine case (see \cite{Akdogan-Demirci-Mukhtarov, Ozkan-Keskin1, Wang, Wei-Wei}) and affine dependence of the square root of the eigenparameter, see \cite{piv}. In particular, transmission conditions of the form
\begin{equation*}
\left[\begin{array}{c} y(0^{+})\\ y'(0^{+})\end{array}\right] = \left[\begin{array}{cc} c & 0\\ h(\lambda) & c^{-1}\end{array}\right] \left[\begin{array}{c} y(0^{-})\\ y'(0^{-})\end{array}\right],
\end{equation*}
where $c \in \mathbb{R}^{+}$ and $h$ is affine in $\lambda$ were considered in \cite{Wei-Wei}, and $c=1$ and $h(\lambda)=i\alpha\sqrt{\lambda}, \alpha>0$ in \cite{piv}.  Recently, the discontinuity condition 
\begin{equation*}
\left[\begin{array}{c} y_{1}(0^{+})\\ y_{2}(0^{+})\end{array}\right] = \left[\begin{array}{cc}c & 0\\ h(\lambda) & c^{-1}\end{array}\right] \left[\begin{array}{c} y_{1}(0^{-})\\ y_{2}(0^{-})\end{array}\right],
\end{equation*}
with $c \in \mathbb{R}^{+}$ and $h$ a polynomial in $\lambda$ was considered in \cite{Keskin} for the Dirac operator
\begin{equation*}
\left[\begin{array}{cc} 0 & 1\\ -1 & 0\end{array}\right] \frac{d Y}{dx} + \left[\begin{array}{cc} p(x) & q(x)\\ q(x) & r(x)\end{array}\right] Y = \lambda Y, \quad Y = \left[\begin{array}{c} y_{1}\\ y_{2}\end{array}\right],
\end{equation*}
with boundary conditions also polynomially dependent on the spectral parameter.

We consider a Sturm-Liouville equation
\begin{equation}\label{SL}
\ell y:=-y'' + qy = \lambda y
\end{equation}
on the intervals $(-a,0)$ and $(0,b)$ with
$y|_{(-a,0)}$, $y|_{(-a,0)}'$, $\ell y|_{(-a,0)} \in L^2(-a,0)$ and 
$y|_{(0,b)}$, $y|_{(0,b)}'$, $\ell y|_{(0,b)} \in L^2(0,b),$ where
$a,b>0$ and $q \in L^2(-a,b)$. We impose boundary conditions 
\begin{align}
  y(-a)\cos\alpha&=y'(-a)\sin\alpha,\label{BC=a}\\
  y(b)\cos\beta&=y'(b)\sin\beta,\label{BC=b}
\end{align}
 where $\alpha \in [0,\pi)$ and $\beta \in (0,\pi]$, and transmission conditions 
 \begin{align}
   y(0^+)\mu(\lambda)&=\Delta' y,\label{TX-0}\\
    y'(0^-)\nu(\lambda)&=\Delta y.\label{TX-1}
\end{align}
Here 
\begin{align}
  \Delta y &= y(0^+) - y(0^-),\nonumber\\ 
  \Delta' y &= y'(0^+) - y'(0^-),\nonumber\\
  \mu(\lambda ) &= -\left(\lambda \eta - \xi -\sum\limits_{i=1}^{N} \frac{b_i^2}{\lambda -c_i}\right),\label{mu}\\
  \nu(\lambda ) &= \lambda \kappa + \zeta -\sum\limits_{j=1}^{M} \frac{a_j^2}{\lambda -d_j},\label{nu}
\end{align}
with $\eta \geq 0$, $\kappa \geq 0$,
\begin{align}
&c_1<c_2<\cdots <c_{N},\label{order-mu}\\
&d_1<d_2<\cdots <d_{M},\label{order-nu}
\end{align}
and $b_i, a_j>0$ for $i=1,\dots,N,$ and $j=1,\dots,M$. We consider $N,M \in \mathbb{N}_0$, with $b_0 = a_0 = 0$.

Further, we will write
\begin{equation}\label{1/mu}
\frac{1}{\mu(\lambda)} = \sigma -\sum\limits_{i=1}^{N'} \frac{\beta_i^2}{\lambda -\gamma _i}, \quad \mbox{if } \eta>0,
\end{equation}
and 
\begin{equation}\label{1/nu}
\frac{1}{\nu(\lambda)} = \tau + \sum\limits_{j=1}^{M'} \frac{\alpha_j^2}{\lambda -\delta _j}, \quad \mbox{if } \kappa >0.
\end{equation}
Here $N', M' \in \mathbb{N}$ and
 \begin{align}
&\gamma_1<\gamma_2<\dots <\gamma_{N'},\label{order-1/mu}\\
&\delta_1<\delta_2<\dots <\delta_{M'},\label{order-1/nu}
\end{align}
with $\beta_i, \alpha_j>0$ for $i =1, \dots, N'$, and $j=1, \dots, M'$.

To the best of our knowledge, this is the first time that spectral properties of (\ref{SL}) with boundary conditions (\ref{BC=a})-(\ref{BC=b}) and transmission conditions dependent on the eigenparameter via general rational Nevanlinna-Herglotz functions (see (\ref{TX-0})-(\ref{TX-1})) have been studied. 

Note that for $\lambda$ a pole of $\mu(\lambda)$ we have that condition (\ref{TX-0}) becomes $y(0^+)=0$ which results in (\ref{TX-1}) becoming $y'(0^-)\nu(\lambda)=-y(0^-)$, resulting in two
separate eigenvalue problems on the intervals $(-a,0)$ and $(0,b)$;
if $\lambda$ is a zero of $\mu(\lambda)$ then (\ref{TX-0}) becomes $\Delta' y=0$, i.e. $y'(0^+)=y'(0^-)$.
Similarly if $\lambda$ is a pole of $\nu(\lambda)$ then (\ref{TX-1}) at $\lambda$ becomes $y'(0^-)=0$ and (\ref{TX-0}) can be expressed as $y(0^+)\mu(\lambda)=y'(0^+)$, again
resulting in separate eigenvalue problems on the intervals $(-a,0)$ and $(0,b)$;
while if $\lambda$ is a zero of $\nu(\lambda)$ then (\ref{TX-1}) becomes $\Delta y=0$, i.e. $y(0^-)=y(0^+)$. 

In Section \ref{sec-intro} the geometric multiplicity of the eigenvalues of (\ref{SL}), (\ref{BC=a})-(\ref{BC=b}) with (\ref{TX-0})-(\ref{TX-1}) is considered and the cases in which the multiplicity can be $2$ are characterized. An example is given to illustrate the cases. 
The Green's function of (\ref{SL}), (\ref{BC=a})-(\ref{BC=b}) with (\ref{TX-0})-(\ref{TX-1}) is given in Section \ref{green}.
A Hilbert space formulation of (\ref{SL}), (\ref{BC=a})-(\ref{BC=b}) with (\ref{TX-0})-(\ref{TX-1}) as a operator eigenvalue problem in
$L^2(-a,b)\bigoplus \C^{N^*} \bigoplus \C^{M^*}$, for suitable $N^*$ and $M^*$, is given in Section \ref{Hilbert}.  In Section \ref{section-self-adjoint} this Hilbert space operator is shown to be self-adjoint. The related Hilbert space resolvent operator is constructed in Section
\ref{section-resolvent}.


\section{Geometric multiplicity\label{sec-intro}}

\begin{lem}\label{simple}
 All eigenvalues of  (\ref{SL})-(\ref{TX-1}) not at poles of $\mu(\lambda)$ or $\nu(\lambda)$ are geometrically simple. In this case the transmission conditions (\ref{TX-0})-(\ref{TX-1}) can be expressed as
$$\left[\begin{array}{c} y(0^+)\\ y'(0^+)\end{array}\right]=T\left[\begin{array}{c} y(0^-)\\ y'(0^-)\end{array}\right], \mbox{ where }T = \left[\begin{array}{cc} 1&\nu(\lambda)\\ \mu(\lambda)&1+ \mu(\lambda)\nu(\lambda) \end{array}\right].$$
 \end{lem}

\begin{proof}
 As $T$ is invertible, imposing (\ref{BC=a}) restricts the solution space of (\ref{SL}) to one dimension.
 \end{proof}

\begin{thm}\label{double}
 The maximum geometric multiplicity of an eigenvalue of  (\ref{SL})-(\ref{TX-1}) is $2$ and such eigenvalues 
 can only occur at poles of $\mu(\lambda)$ or $\nu(\lambda)$.  An eigenvalue $\lambda$ has geometric multiplicity $2$ if and only if: 
\begin{enumerate}
\item[I.] $\lambda$ is a pole of $\mu$ and
 an eigenvalue of  (\ref{SL}) on $(-a,0)$ with 
    boundary conditions (\ref{BC=a}) and $\nu(\lambda)y'(0^-)+y(0^-)=0$, and an
    eigenvalue of (\ref{SL}) on $(0,b)$ with
     boundary conditions $y(0^+)=0$ and (\ref{BC=b}); or
\item[II.] $\lambda$ is a pole of $\nu$ and
  eigenvalue of  (\ref{SL}) on $(-a,0)$ with 
    boundary conditions (\ref{BC=a}) and $y'(0^-)=0$, and an
    eigenvalue of (\ref{SL}) on $(0,b)$ with
     boundary conditions $y(0^+)\mu(\lambda)=y'(0^+)$ and (\ref{BC=b}); or
\item[III.] $\lambda$ is a pole of both $\mu$ and $\nu$ and
  an eigenvalue of  (\ref{SL}) on $(-a,0)$ with 
    boundary conditions (\ref{BC=a}) and $y'(0^-)=0$, and an
    eigenvalue of (\ref{SL}) on $(0,b)$ with
     boundary conditions $y(0^+)=0$ and (\ref{BC=b}).
\end{enumerate}
  \end{thm}
\begin{proof}
 The conclusion that these are only instances in which
 non-simple eigenvalues are possible follows from Lemma \ref{simple}. That the maximum geometric multiplicity is $2$ in the given circumstances follows from the maximum geometric multiplicity of the resulting eigenvalue problems on the intervals $(-a,0)$ and on $(0,b)$ being $1$.
 If $\lambda$ is simultaneously an eigenvalue of the problems on the intervals $(-a,0)$ and on $(0,b)$ with say eigenfunctions $u$ on  $(-a,0)$ and $v$ on $(0,b)$, then extending $u$ and $v$ by zero to $(-a,0)\cup (0,b)$ gives two linearly independent eigenfunctions for (\ref{SL})-(\ref{TX-1}). 
 \end{proof}
 
{\bf Note} Let  
$$\varphi(\lambda)=\frac{\begin{displaystyle}\prod_{j=1}^{m}(s_j-\lambda)\end{displaystyle}}{\begin{displaystyle}\prod_{k=1}^{m}(r_k-\lambda)\end{displaystyle}}$$ where  $r_1<s_1<r_2<s_2<\dots<s_{m-1}<r_m<s_m$ then
$$\varphi(\lambda)=1-\sum_{k=1}^m \frac{K_k}{\lambda-r_k}$$
with
$$K_i=\frac{\begin{displaystyle}\prod_{j=1}^{m}(s_j-r_i)\end{displaystyle}}{\begin{displaystyle}\prod_{k\ne i}(r_k-r_i)\end{displaystyle}}>0.$$
If instead $s_1<r_1<s_2<\dots<r_{m-1}<s_m<r_m$ then
$K_i<0$ for all $i=1,\dots,m$.

\begin{thm}
 For any $N,M\in \N_0$ there are potentials $q\in L^2(-\pi,\pi)$ and parameters
 $c_{1} < c_{2} < \cdots < c_{N}$,
 $d_{1} < d_{2} < \cdots < d_{M}$, $\eta,\kappa \ge 0$
and $b_i, a_j>0$ for $i=1,\dots,N,$ and $j=1,\dots,M$ such that (\ref{SL})-(\ref{TX-1}) with $a=b=\pi$ has precisely $N+M$ double eigenvalues (the maximum number possible).
\end{thm}

\begin{proof}
Assume that $N\le M$.
We take boundary conditions $y(\pm\pi)=0$ and set $q(x)=0$ for $x\in [0,\pi]$.
Now $1^2/4,3^2/4,\dots,(2N-1)^2/4$ are eigenvalues of (\ref{SL}) on $[0,\pi]$ with
 boundary conditions $y(\pi)=0=y'(0^+)$, while
 $1^2,2^2,\dots,N^2$ are eigenvalues of (\ref{SL}) on $[0,\pi]$ with
 boundary conditions $y(\pi)=0=y(0^+)$.
 In particular $\lambda=1^2,2^2,\dots,N^2$ are eigenvalues of  
  (\ref{SL}) on $[0,\pi]$ with
 boundary conditions $y(\pi)=0$ and $\mu(\lambda)y(0^+)=y'(0^+)$
where
$$\mu(\lambda)=\frac{\begin{displaystyle}\prod_{k=1}^{N}\left(\left(k-\frac{1}{2}\right)^2-\lambda\right)\end{displaystyle}}{\begin{displaystyle}\prod_{i=1}^{N}(i^2-\lambda)\end{displaystyle}}=1+\sum_{i=1}^N\frac{b_i^2}
{(\lambda-i^2)}$$
where
$b_i>0, i=1,\dots,N$.

 Let $d_j=(j-1/2)^2$ for $j=1,\dots,N$ and $d_{N+1}<\dots<d_M$ be eigenvalues of 
 (\ref{SL}) on $[0,\pi]$ with
 boundary conditions $y(\pi)=0$ and $\mu(\lambda)y(0^+)=y'(0^+)$ with $d_{N+1}>N^2$.
 Define $\delta_j=j^2,$ for $j=1,\dots,N$, and $\delta_j=(d_{j+1}+d_{j})/2$, $j=N+1,\dots,M-1$
 and $\delta_M>d_M$.
  Let
$$\nu(\lambda)=\frac{\begin{displaystyle}\prod_{j=1}^{M}(\delta_j-\lambda)\end{displaystyle}}{\begin{displaystyle}\prod_{k=1}^{M}(d_k-\lambda)\end{displaystyle}}=1-\sum_{j=1}^M\frac{a_j^2}
{(\lambda-d_j)}$$
where $a_j>0, j=1,\dots,M$.

We now take $q$ on $[-\pi,0)$ to be an $L^2$ potential so that the eigenvalues of (\ref{SL})
on  $[-\pi,0)$ with boundary condition $y(-\pi)=0$ and $y(0^-)=0$ contains the set $\{\delta_1, \dots,\delta_{M}\}$ while  the eigenvalues of (\ref{SL})
on  $[-\pi,0)$ with boundary condition $y(-\pi)=0$ and $y^\prime(0^-)=0$ contains the set $\{d_1, \dots,d_M\}$. 
This is possible via the Gelfand-Levitan theory of inverse spectral problems, see for example
Marcenko \cite[Theorem 3.4.3]{Mar},  since $d_1<\delta_1<d_2<\delta_2<\dots<\delta_{M-1}<d_M<\delta_M$. 

Now $d_1,\dots,d_M$ are poles of $\nu$, so for $\lambda=d_i, i=1,\dots,M,$ the transmission conditions become $y'(0^-)=0$ and $\mu(\lambda)y(0^+)=y'(0^+)$.
By construction of $q$, we have that $d_1,\dots,d_M$ are eigenvalues of the Sturm-Liouville problem on $[-\pi,0]$ with boundary conditions $y(-\pi)=0=y'(0^-)$. Also by construction of
$d_1,\dots,d_M$ they are eigenvalues of the Sturm-Liouville problem on $[0,\pi]$ with boundary 
conditions $\mu(\lambda)y(0^+)=y'(0^+), y(\pi)=0$. Hence $d_1,\dots,d_M$ are double 
eigenvalues of the Sturm-Liouvlle problem on $[-\pi,\pi]$ with transmission condition at $0$ and boundary conditions $y(-\pi)=0=y(\pi)$.

Also $\lambda=\delta_1,\dots,\delta_N$ are poles of $\mu$ so that at these values of $\lambda$ the transmission conditions become $y(0^+)=0, \nu(\lambda)y'(0^-)=-y(0^-)$, but by construction  $\delta_1,\dots,\delta_N$ are zeros of $\nu$. So the transmission conditions become  $y(0^+)=0=y(0^-)$. By construction of $q$, $\delta_1,\dots,\delta_N,$ are eigenvalues of the Sturm-Liouville
problem on $[-\pi,0]$ with boundary conditions $y(-\pi)=0=y(0^-)$. 
However by choice of $\delta_1,\dots,\delta_N$, they are eigenvalues of the Sturm-Liouville problem
on $[0,\pi]$ with boundary conditions $y(0^+)=0=y(\pi)$.

Hence the problem has at least $M+N$ double eigenvalues, which we know to be the maximum possible.     
\end{proof}

We note that using similar methods to those of the above proof, it can be shown that any number of eigenvalues between $0$ and $M+N$ can be constructed to be double.
Due to notational opacity we will only present a proof of the other extreme case, that of no 
double eigenvalues.

\begin{thm}
For any $N,M\in \N_0$ and potentials $q=0$ there are parameters
 $c_{1} < c_{2} < \cdots < c_{N}$,
 $d_{1} < d_{2} < \cdots < d_{M}$, $\eta,\kappa \ge 0$
and $b_i, a_j>0$ for $i=1,\dots,N,$ and $j=1,\dots,M$ such that (\ref{SL})-(\ref{TX-1}) with $a=b=\pi$, $\alpha=0$, $\beta=\pi$,
 has no double eigenvalues.
\end{thm}

\begin{proof}
Let $$\mu(\lambda)=\frac{\begin{displaystyle}\prod_{j=1}^{N}\left(j^2-\lambda\right)\end{displaystyle}}{\begin{displaystyle}\prod_{k=1}^{N}\left(\left(k+\frac{1}{2}\right)^2-\lambda\right)\end{displaystyle}}$$
and
$$\nu(\lambda)=\frac{\begin{displaystyle}\prod_{j=1}^{M}\left(\left(j+\frac{1}{2}\right)^2-\lambda\right)\end{displaystyle}}
{\begin{displaystyle}\prod_{k=1}^{M}\left(k^2-\lambda\right)\end{displaystyle}}.$$
The poles of $\mu$ are $3^2/4,5^2/4,\dots, (2N+1)^2/4$ which are not eigenvalues of the 
Sturm-Liouville problem on $[0,\pi]$ with boundary conditions $y(0^+)=0=y(\pi)$.
The poles of $\nu$ are at $1^2,\dots,M^2$ which are not eigenvalues of the Sturm-Liouville problem on $[-\pi,0]$ with boundary conditions $y(-\pi)=0=y'(0^-)$.
\end{proof}


\section{The Green's function}\label{green}
Let $u_{-}(x;\lambda)$ denote the solution of (\ref{SL}) on $[-a,0)$ satisfying the initial conditions
\begin{equation}
 u_{-}(-a;\lambda) = \sin\alpha\quad\mbox{and}\quad u_{-}'(-a;\lambda) = \cos\alpha, \label{u-}
\end{equation}
and  $v_{+}(x;\lambda)$ denote the solution of (\ref{SL}) on $(0,b]$ satisfying the terminal
conditions
\begin{equation}
 v_{+}(b;\lambda) = \sin\beta\quad\mbox{and} \quad v_{+}'(b;\lambda) = \cos\beta. \label{v+}
\end{equation}
 For $$\lambda\in \Omega :=\{\lambda\in\C\,|\, \mu(\lambda), \, \nu(\lambda) \, \in \C \}$$ we have 
that $u_{-}(x;\lambda)$ and $v_{+}(x;\lambda)$ can, respectively, be extended to solutions $u(x;\lambda)$ and $v(x;\lambda)$  of \eqref{SL}, \eqref{TX-0} and \eqref{TX-1}.
In particular we define $u$ to be $u_-$ on $[-a,0)$ and $v$ to be $v_+$ on $(0,b]$.
Using the 
transfer matrix $T$ as defined in Lemma \ref{simple}, we define $u$ on $(0,b]$ to be the solution of \eqref{SL}  obeying the 
initial condition
\begin{equation}\label{matrix ext u}
\left[\begin{array}{c} u(0^{+};\lambda)\\ u'(0^{+};\lambda)\end{array}\right] = T \left[\begin{array}{c} u(0^{-};\lambda)\\ u'(0^{-};\lambda)\end{array}\right]
\end{equation}
and $v$ on $[-a,0)$ to be the solution of \eqref{SL}  obeying the terminal condition
\begin{equation}\label{matrix ext v}
\left[\begin{array}{c} v(0^{-};\lambda)\\ v'(0^{-};\lambda)\end{array}\right] = T^{-1} \left[\begin{array}{c} v(0^{+};\lambda)\\ v'(0^{+};\lambda)\end{array}\right]. 
\end{equation}
As $T$ has determinant $1$ and
$$\left[\begin{array}{cc} u(0^{+};\lambda)&v(0^{+};\lambda)\\ u'(0^{+};\lambda)&v'(0^{+};\lambda)\end{array}\right] = T \left[\begin{array}{cc} u(0^{-};\lambda)&v(0^{-};\lambda)\\ u'(0^{-};\lambda)&v'(0^{-};\lambda)\end{array}\right]$$ 
it follows that the Wronskian of $u$ and $v$, $W[u,v]$, is constant on $[-a,0)\cup (0,b]$.
We denote $\psi(\lambda) = W[u,v]$.

For a pole of $\mu(\lambda)$ or $\nu(\lambda)$ the above unique extensions are not  available.

\begin{thm}\label{thm-green}
For $\lambda\in \Omega$ not an eigenvalue of  (\ref{SL})-(\ref{TX-1}), the Green's function of (\ref{SL})-(\ref{TX-1}) is given by
\begin{equation}
G(x,t;\lambda) = \begin{cases} \cfrac{u(x;\lambda)v(t;\lambda)}{\psi(\lambda)}, & \mbox{ if  } x<t \mbox{ and } x,t \in [-a,0)\cup(0,b],\\ \cfrac{u(t;\lambda)v(x;\lambda)}{\psi(\lambda)}, & \mbox{ if }t<x \mbox{ and } x,t \in [-a,0)\cup(0,b],\end{cases}
\end{equation}
in the sense that if $h\in L^2(-a,b)$  then
\begin{equation}
g(x;\lambda) = \int_{-a}^{b} G(x,t;\lambda) h(t)dt:={\mathfrak G}_\lambda h
\end{equation}
is a solution of $(\lambda-\ell)g=h$ on $(-a,0)$ and $(0,b)$, such that $g$ obeys the boundary conditions
(\ref{BC=a})-(\ref{BC=b}) and the transmission conditions (\ref{TX-0})-(\ref{TX-1}).
\end{thm}

\begin{proof}
From the above definition of $G$ and $g$, we have
\begin{equation}\label{solution-gf}
g(x)\psi = u(x)\int_{x}^{b} v(t)h(t)\,dt+v(x)\int_{-a}^{x} u(t)h(t)\,dt,
\end{equation}
where for brevity we have suppressed the argument $\lambda$. Differentiating $g$ gives
\begin{equation}\label{solution-gf-der}
g'(x)\psi = u'(x)\int_{x}^{b} v(t)h(t)\,dt+v'(x)\int_{-a}^{x} u(t)h(t)\,dt,
\end{equation}
and a further differentiation gives
\begin{align}\label{solution-gf-l}
g''(x)\psi
&=u''(x)\int_{x}^{b} v(t)h(t)\,dt+v''(x)\int_{-a}^{x} u(t)h(t)\,dt+ h(x)\psi\\
&=(q(x)-\lambda)g(x)\psi+ h(x)\psi 
\end{align}
so  $(\lambda-\ell)g=h$.
Further from (\ref{solution-gf}) and (\ref{solution-gf-der})
\begin{equation*}
\left[\begin{array}{c}g(x) \\ g'(x) \end{array}\right] \psi =\left[\begin{array}{c} u(x)\\ u'(x) \end{array}\right]\int_{x}^{b} v(t)h(t)\,dt+\left[\begin{array}{c} v(x)\\ v'(x) \end{array}\right]\int_{-a}^{x} u(t)h(t)\,dt,
\end{equation*}
from which it follows that 
\begin{equation*}
\left[\begin{array}{c}g(-a) \\ g'(-a) \end{array}\right] \psi =\left[\begin{array}{c} u(-a)\\ u'(-a) \end{array}\right]\int_{-a}^{b} v(t)h(t)\,dt,
\end{equation*}
so (\ref{BC=a}) is obeyed as this condition is obeyed by $u$, and
\begin{equation*}
\left[\begin{array}{c}g(b) \\ g'(b) \end{array}\right] \psi =
\left[\begin{array}{c} v(b)\\ v'(b) \end{array}\right]\int_{-a}^{b} u(t)h(t)\,dt,
\end{equation*}
so (\ref{BC=b}) is obeyed as this condition is obeyed by $v$. Moreover, 
\begin{equation*}
\left[\begin{array}{c}g(0^\pm) \\ g'(0^\pm) \end{array}\right] \psi =\left[\begin{array}{c} u(0^\pm)\\ u'(0^\pm) \end{array}\right]\int_{0}^{b} v(t)h(t)\,dt+\left[\begin{array}{c} v(0^\pm)\\ v'(0^\pm) \end{array}\right]\int_{-a}^{0} u(t)h(t)\,dt,
\end{equation*}
so (\ref{TX-0}) and (\ref{TX-1}) are obeyed as these conditions are obeyed by $u$ and $v$. 
\end{proof}

\begin{thm}\label{thm-green-zero}
For $\lambda$ a pole of $\mu(\lambda)$ or $\nu(\lambda)$ and not an eigenvalue of  (\ref{SL})-(\ref{TX-1}), the Green's function of (\ref{SL})-(\ref{TX-1}) is given by
\begin{equation}
G(x,t;\lambda) = \begin{cases} g^-(x,t;\lambda), &  \mbox{ if } x,t \in [-a,0),\\ 
   g^+(x,t;\lambda), & \mbox{ if } x,t \in (0,b],\\
   0, &\mbox{ if } (x,t)\in (0,b]\times[-a,0) \\ & \mbox{ or } (x,t)\in [-a,0) \times (0,b],\end{cases}
\end{equation}
in the sense that if $h\in L^2(-a,b)$  then
\begin{equation}
g(x;\lambda) = \int_{-a}^{b} G(x,t;\lambda) h(t)dt:={\mathfrak G}_\lambda h
\end{equation}
is a solution of $(\lambda-\ell)g=h$ on $(-a,0)$ and $(0,b)$ such that  $g$ obeys the boundary conditions
(\ref{BC=a})-(\ref{BC=b}) and the transmission conditions (\ref{TX-0})-(\ref{TX-1}).
Here $g^-(x,t;\lambda)$ is the Green's function for the Sturm-Liouville equation \eqref{SL}
on $[-a,0]$ with boundary conditions \eqref{BC=a} and 
\begin{align}
 g'(0^-)\nu(\lambda)&=-g(0^-),\quad\mbox{if $\lambda$ is a pole of $\mu$ but not of $\nu$},\label{g-1}\\
 g'(0^-)&=0,\quad\mbox{if $\lambda$ is a pole of $\mu$ and of $\nu$},\label{g-2}
  \end{align}
and 
$g^+(x,t;\lambda)$ is the Green's function for the Sturm-Liouville equation \eqref{SL}
on $[0,b]$ with boundary conditions \eqref{BC=b} and 
\begin{align}
g(0^+)\mu(\lambda) &=g'(0^+),\quad\mbox{if $\lambda$ is a pole of $\nu$ but not of $\mu$}.\label{g+2}\\
g(0^+)&=0,\quad\mbox{if $\lambda$ is a pole of $\mu$ and of $\nu$},\label{g+1}
\end{align}
\end{thm}

\begin{proof} If $\lambda$ is not an eigenvalue of either (\ref{SL}) on $[-a,0]$ with (\ref{BC=a}) and (\ref{g-1})-(\ref{g-2}), or (\ref{SL}) on $[0,b]$ with (\ref{BC=b}) and (\ref{g+1})-(\ref{g+2}),
then $g(x;\lambda)$ obeys (\ref{BC=a})-(\ref{BC=b}) and (\ref{TX-0})-(\ref{TX-1}), and
$(\lambda-\ell)g=h$ a.e. on $(-a,0)\cup(0,b)$. It thus remains only to show that if $\lambda$
is not an eigenvalue of (\ref{SL})-(\ref{TX-1}) then
 $\lambda$ is not an eigenvalue of either (\ref{SL}) on $[-a,0]$, (\ref{BC=a}) and (\ref{g-1})-(\ref{g-2}) or  (\ref{SL}) on $[0,b]$, (\ref{BC=b}) and (\ref{g+1})-(\ref{g+2}).
 
 If $\lambda$ is an eigenvalue of (\ref{SL}) on $[-a,0]$, (\ref{BC=a}) and (\ref{g-1})-(\ref{g-2}), then let $w$ denote an eigenfunction for this eigenvalue. This eigenfunction, $w$, can then be extended to $(-a,0)\cup (0,b)$ by setting to $0$ on $(0,b)$. It is now apparent that $w$ is an eigenfunction of (\ref{SL})-(\ref{TX-1}), making $\lambda$ an eigenvalue of (\ref{SL})-(\ref{TX-1}). 
 
 Similarly, if $\lambda$ is an eigenvalue (\ref{SL}) on $[0,b]$, (\ref{BC=b}) and (\ref{g+1})-(\ref{g+2}) with eigenfunction say $w$, then $w$, can then be extended to $(-a,0)\cup (0,b)$ by setting to $0$ on $(-a,0)$. It is now apparent that $w$ is an eigenfunction of (\ref{SL})-(\ref{TX-1}) making $\lambda$ an eigenvalue of (\ref{SL})-(\ref{TX-1}). 
 \end{proof}
 

\section{Hilbert space formulation}\label{Hilbert}
We now formulate (\ref{SL}) with boundary conditions (\ref{BC=a})-(\ref{BC=b}) and transmission conditions (\ref{TX-0})-(\ref{TX-1}) as an operator eigenvalue problem in
a Hilbert space $\mathcal{H}$.
For $\eta, \kappa >0$ we set
$\mathcal{H}=L^2(-a,b)\bigoplus \C^{N'}\bigoplus\C^{M'}$ and
\begin{equation}\label{block}
LY := \left[\begin{array}{ccc} \ell & 0 & 0\\
 \boldsymbol{\beta} \Delta' & [\gamma_i] & 0\\
  \boldsymbol{\alpha} \Delta & 0 & [\delta_j]\end{array}\right]\left[ \begin{array}{c} y\\ \mathbf{y}^1\\ \mathbf{y}^2\end{array}\right]
\end{equation}
with domain
$$\mathcal{D}(L)=\left\{  Y \, \left| \,\begin{array}{c}
y_1:=y|_{(-a,0)}\,\mbox{obeys}\,(\ref{BC=a}) \,\mbox{and has}\, y_1,y_1',\ell y_1 \in L^2(-a,0)\\ y_2:=y|_{(0,b)}\,\mbox{obeys}\,(\ref{BC=b})\,\mbox{and has}\, y_2,y_2',\ell y_2 \in L^2(0,b)\\ 
 -y(0^+)+\sigma \Delta' y - \left< {\bf y}^1,\boldsymbol{\beta}\right>=0, \\ 
y'(0^-)-\tau \Delta y - \left<{\bf y}^2, \boldsymbol{\alpha}\right>=0\end{array}\right.\right\},$$
where $Y:= \left[ \begin{array}{c}
y\\
\mathbf{y}^{1}\\ 
\mathbf{y}^{2}
\end{array}\right]$, $[\gamma_i]:={\rm diag}(\gamma_1,\dots,\gamma_{N'})$
$\boldsymbol{\beta} := (\beta_i)$,
${\bf y}^1:= (y_i^1)$,
$[\delta_j] := {\rm diag}(\delta_1,\dots,\delta_{M'})$,
$\boldsymbol{\alpha} := (\alpha_j)$  and ${\bf y}^2 := (y_j^2).$
For later reference, we take the inner product on $\mathcal{H}$ as
\begin{equation*}
\left<Y,Z\right> := (y,z) + \left< \mathbf{y}^1, \mathbf{z}^1 \right> + \left<\mathbf{y}^2, \mathbf{z}^2\right>,
\end{equation*}
where  
$(y,z)=\int_{-a}^{b}y\bar{z}dx$ and $\left<\cdot ,\cdot \right>$ denotes Euclidean inner product.

Replacement of $\boldsymbol{\beta},  \boldsymbol{\gamma}, \sigma, y(0^+), \Delta'y$ and $N'$ by $\boldsymbol{b}, \boldsymbol{c}, -\xi, -\Delta'y, y(0^+)$ and $N$ in the specification of $\mathcal{H}$, $L$ and $\mathcal{D}(L)$ above and the results of this section, below, yields the case of $\eta=0$, while replacement of $\boldsymbol{\alpha},  \boldsymbol{\delta}, \tau, y'(0^-), \Delta  y$ and $M'$ by $\boldsymbol{a}, \boldsymbol{d}, -\zeta, -\Delta y, y'(0^-)$ and $M$ yields the case of $\kappa=0$.

\begin{thm}
 The eigenvalue problems $LY = \lambda Y$ 
 and (\ref{SL}) with boundary conditions (\ref{BC=a})-(\ref{BC=b}) and transmission conditions (\ref{TX-0})-(\ref{TX-1}) are equivalent in the sense that
 $\lambda$ is an eigenvalue of  $LY = \lambda Y$ with eigenvector $Y$ if and only if 
 $\lambda$ is an eigenvalue, with eigenfunction $y$, of
 (\ref{SL}) with boundary conditions (\ref{BC=a})-(\ref{BC=b}) and transmission conditions
 (\ref{TX-0})-(\ref{TX-1}).
 Here, for $\eta,\kappa>0$, 
\begin{equation}\label{e-vector-1}
{\bf y}^1 = (\lambda I-[\gamma_i])^{-1}\boldsymbol{\beta}\Delta' y 
\end{equation}
if $\lambda \neq \gamma_i$ for all $i=1,\dots,N'$ while if $\lambda = \gamma_I$ for some $I \in \{1,\dots,N'\}$  then 
\begin{equation}\label{e-vector-i-poles} {\bf y}^1 = \frac{-y(0^+)}{\beta_I} \boldsymbol{e}^I,
\end{equation}
with $\boldsymbol{e}^I$  the vector in $\R^{N'}$ with all entries $0$ except the $I^{th}$ which is $1$, and  
\begin{equation}\label{e-vector-2}
{\bf y}^2 = (\lambda I-[\delta_j])^{-1}\boldsymbol{\alpha}\Delta y, 
\end{equation}
if $\lambda \neq \delta_j$ for all $j=1,\dots,M'$,
while if $\lambda = \delta_J$ for some $J \in \{1,\dots,M'\}$ then 
\begin{equation}\label{e-vector-j-poles}
 {\bf y}^2 = \frac{y'(0^-)}{\alpha_J} \boldsymbol{e}^J,
 \end{equation}
 with $\boldsymbol{e}^J$  the vector in $\R^{M'}$ with all entries $0$ except the $J^{th}$ which is $1$.
 The geometric multiplicity of $\lambda$ as an eigenvalue of $L$ is the same as the geometric multiplicity
 of $\lambda$ as an eigenvalue of (\ref{SL})-(\ref{TX-1}).
\end{thm}

\begin{proof}
If $LY=\lambda Y$, then $\ell y = \lambda y$, where $y|_{(-a,0)}$, $y'|_{(-a,0)}$, $\ell y|_{(-a,0)} \in L^{2}(-a,0)$ and $y|_{(0,b)}, y'|_{(0,b)}, \ell y|_{(0,b)} \in L^{2}(0,b)$. Moreover $y$ obeys (\ref{BC=a}) and (\ref{BC=b}). 
By definition of $LY$, 
$\gamma_{i} y^{1}_{i} + \beta_{i}\Delta' y = \lambda y^{1}_{i}$  for all $i$, the domain conditions give
$-y(0^+)+\sigma \Delta' y - \left< {\bf y}^1,\boldsymbol{\beta}\right>=0.$ 
Hence if $\lambda \neq \gamma_{i}$ for all $i$, then 
$y^{1}_{i} = \frac{\beta_{i}}{\lambda - \gamma_{i}}\Delta'y$ which with the domain condition gives
$$y(0^{+}) = \left[\sigma -\sum\limits_{i=1}^{N'} \frac{\beta_i^2}{\lambda -\gamma_i}\right] \Delta' y.$$
If $\lambda = \gamma_{I}$ then $\gamma_{i} y^{1}_{i} + \beta_{i}\Delta' y = \lambda y^{1}_{i}$ gives $\Delta' y = 0$ and $y^1_{i}=0$ for all $i\ne I$, which together with the domain condition gives $y^{1}_{I}=  \frac{-y(0^{+})}{\beta_{I}}$.
Hence $y$ obeys (\ref{TX-0}). 

Similarly
$\delta_{j}y^{2}_{j} + \alpha_{j}\Delta y = \lambda y^{2}_{j}$ for all $j$, giving $y^{2}_{j} = \frac{\alpha_{j}}{\lambda - \delta_{j}}\Delta y$ if $\lambda \neq \delta_{j}$ for all $j$, 
and $\Delta y = 0$ if for some $J$, $\lambda = \delta_{J}$ and $y_j^2=0$ for all $j\ne J$.
The domain condition $y'(0^-)-\tau \Delta y - \left<{\bf y}^2, \boldsymbol{\alpha}\right>=0$ in the case of $\lambda \neq \delta_{j}$ for all $j$ now gives that
\begin{equation}
y'(0^{-}) = \left[\tau + \sum\limits_{j=1}^{M'} \frac{\alpha_j^2}{\lambda -\delta_j}\right] \Delta y,\label{tx-2-nov}
\end{equation}
while for $\lambda = \delta_{J}$ the domain condition forces
$y^{2}_{J}=\frac{y'(0^{-})}{\alpha_{J}}$ from which (\ref{TX-1}) follows.

Hence, the eigenvalues of $L$ are eigenvalues of (\ref{SL}) with boundary conditions (\ref{BC=a})-(\ref{BC=b}) and transmission conditions (\ref{TX-0})-(\ref{TX-1}), with eigenfunction $y = \left[Y\right]_{0}$ (i.e. the functional component of $Y$).

For the converse, let $\lambda$ be an eigenvalue, with eigenfunction $y$, of (\ref{SL}) with boundary conditions (\ref{BC=a})-(\ref{BC=b}) and transmission conditions (\ref{TX-0})-(\ref{TX-1}). Here $\ell y=\lambda y$ on $(-a,0)\cup (0,b)$, with $y|_{(-a,0)}, y'|_{(-a,0)}, \ell y|_{(-a,0)} \in L^{2}(-a,0)$ and $y|_{(0,b)}, y'|_{(0,b)}, \ell y|_{(0,b)} \in L^{2}(0,b)$. Define $Y$ as given in (\ref{e-vector-1})-(\ref{e-vector-j-poles}). 

For $\lambda \neq \gamma_i$ for all $i$, from the form given for ${\bf y}^1$,  
$$[LY]_1=[\boldsymbol{\beta} \Delta' \,,\, [\gamma_i] \,,\, 0 ]Y=
\boldsymbol{\beta} \Delta'y +[\gamma_i]{\bf y}^1
= \lambda {\bf y}^1,$$
and since $y$ obeys (\ref{TX-0}),
$$\left<{\bf y}^1,\boldsymbol{\beta}\right>=\sum\limits_{i=1}^{N'}\frac{\beta _i^2}{\lambda -\gamma_i}\Delta' y= -y(0^+) +\sigma \Delta' y,$$
so the domain condition for ${\bf y}^1$ is obeyed. 

For $\lambda = \gamma_{I}$, for some $I\in \{1,...,N'\}$, we have 
$\Delta'y=0$ so
$$[LY]_1=[\boldsymbol{\beta} \Delta' \,,\, [\gamma_i] \,,\, 0 ]Y=
[\gamma_i]{\bf y}^1= \lambda {\bf y}^1,$$ 
since $y_i^1=0$ for $i\ne I$.
Also, since $y^1_I=-y(0^+)/\beta_I$,
$$\left<{\bf y}^1,\boldsymbol{\beta}\right>=-y(0^+)= -y(0^+) +\sigma \Delta' y,$$
and the domain condition relating to  ${\bf y}^1$ is obeyed.

If $\lambda \neq \delta_{j}$ for all $j = 1, \dots, M'$, from (\ref{e-vector-2}), 
$$[LY]_2= [\boldsymbol{\alpha} \Delta\,,\, 0 \,,\, [\delta_j]]Y=\boldsymbol{\alpha} \Delta y+ [\delta_j]{\bf y}^2=\lambda{\bf y}^2,$$
while  (\ref{TX-1}), (\ref{1/nu}) and (\ref{e-vector-2}) combined give that the domain condition
$$y'(0^-)-\tau \Delta y - \left<{\bf y}^2, \boldsymbol{\alpha}\right>=
y'(0^-)-\tau \Delta y -\sum\limits_{j=1}^{M'} \frac{\alpha_j^2}{\lambda - \delta_j} \Delta y=
0$$
is satisfied.

For $\lambda = \delta_{J}$, for some $J \in \{1,...,M'\}$, from (\ref{TX-1}) and (\ref{1/nu}) we have $\Delta y=0$ which, together with (\ref{e-vector-j-poles}), gives
$$[LY]_2= [\boldsymbol{\alpha} \Delta\,,\, 0 \,,\, [\delta_j]]Y= [\delta_j]{\bf y}^2=\lambda{\bf y}^2,$$
while $\Delta y=0$ and (\ref{e-vector-j-poles}) give that the domain condition
$$y'(0^-)-\tau \Delta y - \left<{\bf y}^2, \boldsymbol{\alpha}\right>=0$$
is satisfied.

Next we consider the correspondence of geometric multiplicities. If $\lambda$ is an eigenvalue of 
 (\ref{SL})-(\ref{TX-1}) with eigenfunctions $y^{[1]},\dots,y^{[k]}$ which are linearly independent then
 the vectors $Y^{[1]},\dots,Y^{[k]}$ as given by (\ref{e-vector-1})-(\ref{e-vector-j-poles}) are linearly independent eigenvectors of $L$ for the eigenvalue $\lambda$. Hence the geometric multiplicity of $\lambda$ 
 as an eigenvalue of $L$ is at least as large as the  geometric multiplicity of $\lambda$ 
 as an eigenvalue  (\ref{SL})-(\ref{TX-1}).

 If   $Y^{[1]},\dots,Y^{[k]}$ are linearly independent eigenvectors of $L$ for the eigenvalue
 $\lambda$ then, from the first part of this theorem, the functional components  $y^{[1]}=[Y^{[1]}]_0,\dots,y^{[k]}=[Y^{[k]}]_0$ are eigenvectors of  (\ref{SL})-(\ref{TX-1}) for the eigenvalue
 $\lambda$. It remains only to prove their linear independence. 
If  there are $\rho_1,\dots,\rho_k$, not all zero, with  
 $$0=\sum_{n=1}^k\rho_n y^{[n]},$$
 then from (\ref{e-vector-1})-(\ref{e-vector-j-poles})
 $$\mathbf{0}= \sum_{n=1}^k\rho_n Y^{[n]},$$
 contradicting the linear independence of $Y^{[1]},\dots,Y^{[k]}$. Hence 
 $y^{[1]},\dots,y^{[k]}$ are linearly independent and the geometric multiplicity of $\lambda$ as
 an eigenvalue of $L$ coincides with its geometric multiplicity as an eigenvalue of (\ref{SL})-(\ref{TX-1}).
\end{proof}


\section{Self-adjointness}\label{section-self-adjoint}

In this section we show that $L$ is a self-adjoint (densely defined) operator in $\mathcal{H}$.

\begin{thm}\label{self adjoint}
 The operator $L$ is self-adjoint in $\mathcal{H}$. 
\end{thm}

\begin{proof}
We present the proof for the case of $\eta,\kappa>0$, the proofs for the other cases being similar.

We begin by showing that $\mathcal{D}(L)$ is dense in $\mathcal{H}$.
Let $A,B,C,D\in \mathbb{R}$, 
\begin{equation}\label{w_m}
F = \left[ \begin{array}{c} f\\ {\bf f^1}\\ {\bf f^2}\end{array}\right]\in \mathcal{H} \quad\mbox{and}\quad 
W = \left[ \begin{array}{c} w\\ {\bf f^1}\\ {\bf f^2}\end{array}\right]
\end{equation}     
where $w$ is $C^\infty$ on $[-a,0)$ and $(0,b]$ with $w(-a)=w'(-a)=0=w(b)=w'(b)$ so that $w|_{[-a,0)}$ has an extension to a 
function $C^\infty$ on $[-a,0]$, $w|_{(0,b]}$ has an extension to a 
function $C^\infty$ on $[0,b]$ and
\begin{align*}
w(0^-)&= (\sigma C-A-1)\left<{\bf f^1},\boldsymbol{\beta}\right> +(\sigma D-B)\left<{\bf f^2},\boldsymbol{\alpha}\right>,\\
w(0^+)&=(\sigma C-1)\left<{\bf f^1},\boldsymbol{\beta}\right> +\sigma D\left<{\bf f^2},\boldsymbol{\alpha}\right>, \\
w'(0^-)&= \tau A\left<{\bf f^1},\boldsymbol{\beta}\right> +(\tau B+1)\left<{\bf f^2},\boldsymbol{\alpha}\right>,\\
w'(0^+)&=(\tau A+C)\left<{\bf f^1},\boldsymbol{\beta}\right> +(\tau B+D+1)\left<{\bf f^2},\boldsymbol{\alpha}\right>,
\end{align*}
then $W \in \mathcal{D}(L)$ and
\begin{align}
\Delta w&= A\left<{\bf f^1},\boldsymbol{\beta}\right> +B\left<{\bf f^2},\boldsymbol{\alpha}\right>,\label{delta-1}\\
\Delta' w&=C\left<{\bf f^1},\boldsymbol{\beta}\right> +D\left<{\bf f^2},\boldsymbol{\alpha}\right>.\label{delta-2}
\end{align}

As $q\in L^2(-a,b)$ it follows that 
$\left(C^\infty_0(-a,0)\bigoplus C^\infty_0(0,b)\right)\bigoplus \{ {\bf 0} \}\bigoplus \{ {\bf 0 }\}\subset \mathcal{D}(L)$.
Here, $C^\infty_0(-a,0)\bigoplus C^\infty_0(0,b)$ is dense in $L^2(-a,b)$ so  
there is a sequence $\{g_n\}\subset C^\infty_0(-a,0)\bigoplus C^\infty_0(0,b)$ with $g_n \to f-w$ in norm. Here,
$G_n:=\left[ g_n\    {\bf 0}\  {\bf 0}\right]^T\in \mathcal{D}(L)$
and thus $W+G_n\in \mathcal{D}(L)$. Now, $W+G_n \to F$ in norm as $n\to\infty$ giving that $\mathcal{D}(L)$ is dense in $\mathcal{H}$.

We  now show that $L$ is symmetric.
 Let $F, G\in\mathcal{D}(L)$, then the functional components $f$ and $g$ of $F$ and
$G$ respectively obey
\begin{equation*}
(\ell f,g)- (f,\ell g) = (-f'\bar{g} + f\bar{g}')(0^-)+ (f'\bar{g} - f\bar{g}')(0^+).
\end{equation*}
Moreover, the vector components satisfy
\begin{align*}
\left<\boldsymbol{\beta}\Delta' f+ [\gamma_i] {\bf f}^1, {\bf g}^1\right>- \left<{\bf f}^1, \boldsymbol{\beta}\Delta' g + [\gamma_i] {\bf g}^1\right> &= \left< \boldsymbol{\beta}\Delta' f, {\bf g}^1\right> - \left<{\bf f}^1, \boldsymbol{\beta}\Delta'  g\right>,\\
 \left< \boldsymbol{\alpha}\Delta f+ [\delta_j] {\bf f}^2, {\bf g}^2\right>- \left<{\bf f}^2, \boldsymbol{\alpha}\Delta g + [\delta_j] {\bf g}^2\right> &= \left< \boldsymbol{\alpha}\Delta f, {\bf g}^2\right> - \left<{\bf f}^2, \boldsymbol{\alpha}\Delta g\right>,
 \end{align*}
where the domain conditions give
\begin{align*}
\left< \boldsymbol{\beta}\Delta' f, {\bf g}^1\right> - \left<{\bf f}^1, \boldsymbol{\beta}\Delta' g\right> &= \Delta' f [ -\bar{g}(0^+) + \sigma\Delta'\bar{g}] - \Delta' \bar{g}[-f(0^+) + \sigma\Delta' f], 
\\
\left< \boldsymbol{\alpha}\Delta f, {\bf g}^2\right> - \left<{\bf f}^2, \boldsymbol{\alpha}\Delta g\right> &= \Delta f [\bar{g}'(0^-) - \tau \Delta\bar{g}] - \Delta \bar{g} [f'(0^-) - \tau\Delta f].
\end{align*}
Hence
\begin{eqnarray*}
\left< \boldsymbol{\beta}\Delta' f+ [\gamma_i]{\bf f}^1, {\bf g}^1\right>- \left<{\bf f}^1, \boldsymbol{\beta}\Delta' g +[\gamma_i] {\bf g}^1\right> &=& f(0^+)\Delta'\bar{g} - \bar{g}(0^+)\Delta' f ,\\
\left< \boldsymbol{\alpha}\Delta f+ [\delta_j] {\bf f}^2, {\bf g}^2\right>- \left<{\bf f}^2, \boldsymbol{\alpha}\Delta g + [\delta_j] {\bf g}^2\right> &=& \bar{g}'(0^-)\Delta f - f'(0^-)\Delta\bar{g}.
\end{eqnarray*}
Direct computation gives
\begin{align*}
 (f'\bar{g}-f\bar{g}')(0^-) &-(f'\bar{g} - f\bar{g}')(0^+)\\ &=
f(0^+)\Delta'\bar{g} - \bar{g}(0^+)\Delta' f
+ \bar{g}'(0^-)\Delta f -f'(0^-)\Delta\bar{g}, 
\end{align*}
thus $\left<LF, G\right> - \left<F, LG\right> = 0$ and so $L$ is symmetric, giving $\mathcal{D}(L)\subset \mathcal{D}(L^*)$.

To show that $L$ is self-adjoint it remains only to verify that $\mathcal{D}(L^*)\subset \mathcal{D}(L)$. 
Let $G\in \mathcal{D}(L^*)$ then
$\left<LF,G\right>=\left<F,L^*G\right>$ for all $F\in \mathcal{D}(L)$, and the map $F\mapsto \left<F,L^*G\right>$ defines a continuous linear functional on $\mathcal{H}$.
 Hence, the map $F\mapsto \left<LF,G\right>$ is a continuous linear functional on $\mathcal{H}$ restricted to the 
dense subspace $\mathcal{D}(L)$. In particular, there is $k\ge 0$ so that for all $F\in \left( C^\infty_0(-a,0)\bigoplus\{0\}\right)\bigoplus\{{\bf 0}\}\bigoplus\{{\bf 0} \}$
we have that  
\begin{eqnarray}
\left|\int_{-a}^0 f''\left(-\overline{g}+\int_{-a}^x\int_{-a}^t q\overline{g}\,d\tau\,dt\right)\,dx\right|\le k\|f\|_2,\label{bound}
\end{eqnarray}
for all $f\in C^\infty_0(-a,0)$. Hence, see \cite[Chapters 1 \& 2]{agmon},
\begin{eqnarray}
g-\int_{-a}^x\int_{-a}^t qg\,d\tau\,dt\in H^2(-a,0).\label{recursive}
\end{eqnarray}
We note here that $qg\in L^1(-a,0)$, giving that $\int_{-a}^t qg\,d\tau\,dt \in L^2(-a,0)$.
Hence, $g\in H^1(-a,0)$ and differentiating (\ref{recursive}) gives
\begin{eqnarray}
g'-\int_{-a}^x qg\,d\tau\,\in H^1(-a,0).\label{recursive-1}
\end{eqnarray}
Thus $g''$ exists as a weak derivative and is in $L^1(-a,0)$.
Applying the above in (\ref{bound}) gives
\begin{eqnarray}
\left|\int_{-a}^0 f\left(-\overline{g}''+q\overline{g}\right)\,dx\right|\le k\|f\|_2,\label{bound-1}
\end{eqnarray}
and hence $\ell^*g=\ell g$ exists in $L^2(-a,0)$.
Similarly, we obtain $g,g', \ell^*g=\ell g$ exists in $L^2(0,b)$.
Thus
$g\in H^2(-a,0)\bigoplus H^2(0,b)$ with $\ell^*g=\ell g \in L^2(-a,b)$.
In the light of the above, taking $F=[f\ \bf{0}\ \bf{0}]^T$ and varying $f$ through $C^\infty[-a,0)\bigoplus C^\infty(0,b]$ obeying (\ref{BC=a}), (\ref{BC=b}) and having $f^{(m)}(\pm 0)=0$ for all $m=0,1,2,\dots,$
 we obtain that $g$ obeys (\ref{BC=a}) and (\ref{BC=b}). 

Now let $W$ be as in (\ref{w_m}),
then we have that
\begin{align*}
&\int_{-a}^b \ell w\, \overline{g}\,dx+\left< \boldsymbol{\beta}\Delta' w + [\gamma_i] {\bf f}^1, [G]_1 \right> +\left<\boldsymbol{\alpha}\Delta w + [\delta_j]{\bf f}^2, [G]_2\right>\\
&=
\left<LW,G\right>\\
&=\left<W,L^*G\right>\\
&=\int_{-a}^b w\,\ell\overline{g}\,dx+ \left<{\bf f}^1, [L^*G]_1\right> +\left<{\bf f}^2, [L^*G]_2\right>.
\end{align*}

Applying integration by parts to the pair of integrals in the above expression we have
\begin{align*}
&(w'\overline{g})(0^+)-(w'\overline{g})(0^-)
+\left< \boldsymbol{\beta}\Delta' w + [\gamma_i] {\bf f}^1, [G]_1 \right> +\left<\boldsymbol{\alpha}\Delta w + [\delta_j]{\bf f}^2, [G]_2\right>
\\
&=(w\overline{g}')(0^+)-(w\overline{g}')(0^-)
+\left< {\bf f}^1, [L^*G]_1 \right>
+\left< {\bf f}^2, [L^*G]_2 \right>.
\end{align*}
Using (\ref{delta-1}), (\ref{delta-2}) and the domain conditions obeyed by $W$ to simplify the above we obtain
\begin{equation}
0=\left< {\bf f}^1, V_1-(AS_2+CS_1)\boldsymbol{\beta}\right>
+\left< {\bf f}^2, V_2-(BS_2+DS_1)\boldsymbol{\alpha} \right>,\label{2018-sa}
\end{equation}
where
\begin{align*}
V_1&=-\Delta'g\boldsymbol{\beta}-[\gamma_i][G]_1+[L^*G]_1,\\
V_2&=-\Delta g\boldsymbol{\alpha}-[\delta_j][G]_2+[L^*G]_2,\\
S_1&=g(0^+)-\sigma \Delta'g+\left<[G]_1,\boldsymbol{\beta}\right>,\\
S_2&=-g'(0^-)+\tau \Delta g+\left<[G]_2,\boldsymbol{\alpha}\right>.
\end{align*}
Varying ${\bf f}^1$ and ${\bf f}^2$ in (\ref{2018-sa}) gives
\begin{align}
0&=V_1-(AS_2+CS_1)\boldsymbol{\beta}, \label{2018-b}\\
0&=V_2-(BS_2+DS_1)\boldsymbol{\alpha}. \label{2018-c}
\end{align}
However (\ref{2018-b}) and (\ref{2018-c}) hold for all $A,B,C,D\in \mathbb{R}$, giving that $S_1$, $S_2$, $V_1$, $V_2=0$.
In particular $S_1=0=S_2$ give $G\in \mathcal{D}(L)$.
\end{proof}


\section{The resolvent operator}\label{section-resolvent}

The block operator form for $L$ given in (\ref{block}) together with the Green's function
for (\ref{SL}), (\ref{BC=a})-(\ref{BC=b}) with (\ref{TX-0})-(\ref{TX-1}) given in Theorems \ref{thm-green} and \ref{thm-green-zero} along with the domain conditions of $L$ lead to the expressions for the resolvent of $L$ given in the following theorems. In the interest of readability the construction of the resolvent will be done in two parts.

\begin{thm}\label{resolvent-1}
For $\lambda$ not an eigenvalue of $L$, $h\in L^2(-a,b)$ and $\lambda\ne \gamma_i,\delta_j$ for all $i,j$, if $\eta>0$ and $\kappa>0$, we have that
\begin{align}
(\lambda - L)^{-1}\left[\begin{array}{c}h\\ 0\\ 0\end{array}\right]
 = \left[
\begin{array}{c} {\mathfrak G}_\lambda h \\ (\lambda I- [\gamma_i])^{-1}\boldsymbol{\beta}\Delta' {\mathfrak G}_\lambda h
\\ (\lambda I- [\delta_j])^{-1}\boldsymbol{\alpha}\Delta {\mathfrak G}_\lambda h 
\end{array}
\right]=:G_0h=:\left[\begin{array}{c}f\\  {\mathbf f}^1\\  {\mathbf f}^2\end{array} \right],\label{2018-1}
\end{align} 
where for $\lambda=\gamma_I$ we replace $(\lambda I- [\gamma_i])^{-1}\boldsymbol{\beta}\Delta' {\mathfrak G}_\lambda h$  by 
$-\frac{1}{\beta_I}{\mathfrak G}_\lambda h(0^+){\bf e}^{I}$, and for $\lambda=\delta_J$ we replace
$(\lambda I- [\delta_j])^{-1}\boldsymbol{\alpha}\Delta {\mathfrak G}_\lambda h$  by 
$\frac{1}{\alpha_J}({\mathfrak G}_\lambda h)'(0^-){\bf e}^J$.
For $\eta=0$ replace $\boldsymbol{\beta}\Delta' {\mathfrak G}_\lambda h$, $-\frac{1}{\beta_I}{\mathfrak G}_\lambda h(0^+)$ and $\gamma_i$ by 
$\boldsymbol{b} {\mathfrak G}_\lambda h(0^+)$, $\frac{1}{b_I}\Delta'{\mathfrak G}_\lambda h$ and $c_i$. For $\kappa=0$,
replace $\boldsymbol{\alpha}\Delta{\mathfrak G}_\lambda h$, $\frac{1}{\alpha_J}({\mathfrak G}_\lambda h)'(0^-)$ and $\delta_j$  by 
$\boldsymbol{a}({\mathfrak G}_\lambda h)'(0^-)$, $-\frac{1}{a_J}\Delta{\mathfrak G}_\lambda h$ and $d_j$.
\end{thm}

\begin{proof}
We present the proof for $\eta,\kappa>0$, the other cases are similar with the symbol replacements noted above. 

We begin by showing that $G_0h\in \mathcal{D}(L)$.
From the definition of ${\mathfrak G}_\lambda h$ it follows that $f|_{(-a,0)}$, $f|_{(-a,0)}'$, $\ell f|_{(-a,0)} \in L^2(-a,0)$, $f|_{(0,b)}$, $f|_{(0,b)}'$, $\ell f|_{(0,b)} \in L^2(0,b),$ and that $f$ obeys (\ref{BC=a})  and  (\ref{BC=b}).
Moreover,
\begin{equation}\label{G 1}
f(0^+) ={\mathfrak G}_\lambda h(0^+)= \left[\sigma - \boldsymbol{\beta}^T (\lambda - [\gamma_i])^{-1}\boldsymbol{\beta}\right]\Delta' {\mathfrak G}_\lambda h=
\sigma\Delta'f - \left<\boldsymbol{f}^1,\boldsymbol{\beta}\right>,
\end{equation}
if $\lambda \neq \gamma_i$ for all $i$, and 
\begin{equation}\label{G 2}
f'(0^-)=({\mathfrak G}_\lambda h)'(0^-)= \left[\tau  + \boldsymbol{\alpha}^T(\lambda - [\delta_j])^{-1}\boldsymbol{\alpha} \right]\Delta {\mathfrak G}_\lambda h= \tau\Delta f  + \left<\boldsymbol{f}^2,\boldsymbol{\alpha}\right>
\end{equation}
if $\lambda \neq \delta_j$ for all $j$.
Here (\ref{G 1}) and (\ref{G 2}) follow from the definition of the Green's operator in Theorem  \ref{thm-green} and Theorem \ref{thm-green-zero}. 
For $\lambda=\gamma_I$, it follows from the definition of the Green's operator that
$\Delta'{\mathfrak G}_\lambda h=0$. Thus setting $\boldsymbol{f}^1=-\frac{1}{\beta_I}{\mathfrak G}_\lambda h(0^+){\bf e}^I$ we have that (\ref{G 1}) is replaced by
\begin{equation}\label{G 3}
f(0^+)=-\left<\boldsymbol{f}^1,\boldsymbol{\beta}\right>=\sigma\Delta'f-\left<\boldsymbol{f}^1,\boldsymbol{\beta}\right>.
\end{equation}
For $\lambda = \delta_J$ we have that $\Delta \mathfrak{G}_\lambda h = 0$ by definition of $\mathfrak{G}_\lambda h$, and setting ${\bf f}^2 = \frac{1}{\alpha_J} (\mathfrak{G}_\lambda h)'(0^-){\bf e}^J$ replaces (\ref{G 2}) by
\begin{equation}\label{G 4}
f'(0^-) = \left\langle {\bf f}^2 , \boldsymbol{\alpha}\right\rangle = \tau \Delta f + \left\langle {\bf f}^1, \boldsymbol{\alpha}\right\rangle.
\end{equation}

Thus $G_0h\in \mathcal{D}(L)$.
The formal verification that 
$(\lambda - L)G_0h=[h,0,0]^T$ is straight forward.
\end{proof}

\begin{thm}\label{resolvent-gamma}
If $\lambda$ is not an eigenvalue of $L$ then we have that 
\begin{equation}
(\lambda-L)^{-1}\left[\begin{array}{c} 0\\ \mathbf{h}^1\\ \mathbf{h}^2\end{array} \right]
=\left[\begin{array}{c}f\\ \mathbf{f}^1\\ \mathbf{f}^2\end{array} \right],\label{baw-1}
\end{equation}
where 
\begin{equation}\label{gen}
f=A\chi_{[-a,0)}u_-+B\chi_{(0,b])}v_+.
\end{equation}
Here
\begin{equation}
\begin{bmatrix} A \\ B \end{bmatrix} = \Lambda \begin{bmatrix}\left<{\bf h}^1, {\bf p} \right> \\\left<{\bf h}^2, {\bf q} \right>\end{bmatrix}
\label{ABcoeff}
\end{equation}
with
\begin{equation}
{\bf p}=\begin{cases} (\lambda I- [c_i])^{-1}{\bf b}, & \eta =0, \\
					\mu(\lambda)(\lambda I-[\gamma_i])^{-1}\boldsymbol{\beta}, & \eta >0, \end{cases} \label{p}
\end{equation}
\begin{equation}
{\bf q}=\begin{cases} (\lambda I - [d_j])^{-1}{\bf a}, & \kappa =0, \\
					\nu(\lambda)(\lambda I-[\delta_j])^{-1}\boldsymbol{\alpha}, & \kappa >0, \end{cases} \label{q}
\end{equation}
\begin{equation*}
\Lambda= \frac{1}{D}
\begin{bmatrix}
-v_+(0^+) & -v_+'(0^+)+\mu(\lambda)v_+(0^+) \\ -u_-(0^-)-\nu(\lambda) u_-'(0^-)& -u_-'(0^-)
\end{bmatrix} 
\end{equation*}
and $D=u_-'(0^-)v_+(0^+)-(v_+'(0^+)-\mu(\lambda)v_+(0^+))(u_-(0^-)+\nu(\lambda)u_-'(0^-))$.
Therefore $f, {\bf f}^1$ and ${\bf f}^2$ are given uniquely. 
At poles, $f$ can be determined by residue calculations. 
Also, 
\begin{align}
-\boldsymbol{\beta}\Delta' f + (\lambda I-[\gamma_i]){\bf f}^{1}&={\bf h}^{1}, \quad \eta >0, \label{OP1}\\
-{\bf b} f(0^+) + (\lambda I-[c_i]){\bf f}^{1}&={\bf h}^{1}, \quad \eta = 0, \label{OP3}\\
-\boldsymbol{\alpha}\Delta f + (\lambda I-[\delta_j]){\bf f}^{2}&={\bf h}^{2}, \quad \kappa >0, \label{OP2}\\
-{\bf a}f'(0^-) + (\lambda I-[d_j]){\bf f}^{2}&={\bf h}^{2}, \quad \kappa = 0. \label{OP4}
\end{align}
\end{thm}

\begin{proof}
 The general solution of $(\lambda-\ell)f=0$ on $[-a,0)\cup (0,b]$ obeying boundary conditions
(\ref{BC=a}) and (\ref{BC=b}) is given by (\ref{gen}), where $u_-, v_+$ 
are as defined at the beginning of Section \ref{green} and $u_-, v_+  \not \equiv 0$ obey (\ref{u-}) and (\ref{v+}) respectively. By (\ref{baw-1}), the operator conditions
(\ref{OP1})-(\ref{OP4}) follow.

From the domain of the operator $L$ we obtain the domain conditions
\begin{align}
 -f(0^+)+\sigma \Delta' f - \left< {\bf f}^1,\boldsymbol{\beta}\right>&=0, \quad \eta >0, \label{DC1} \\ 
 \Delta' f - \xi f(0^+) - \left< {\bf f}^1,{\bf b}\right> &=0, \quad \eta =0, \label{DC3} \\
f'(0^-)-\tau \Delta f - \left<{\bf f}^2, \boldsymbol{\alpha}\right>&=0, \quad \kappa >0, \label{DC2} \\
-\Delta f+\zeta f'(0^-) - \left<{\bf f}^2, {\bf a}\right>&=0, \quad \kappa =0. \label{DC4}
\end{align}

\noindent $\underline{ \eta >0}:$ Suppose that $\lambda \neq \gamma_i$ for all $i.$ Substituting ${\bf f}^1$ from (\ref{OP1}) into (\ref{DC1}) results in
\begin{equation*}
-f(0^+)+\sigma \Delta' f - \left< (\lambda I-[\gamma_i])^{-1}({\bf h}^1+\boldsymbol{\beta} \Delta' f),\boldsymbol{\beta}\right> =0.
\end{equation*}
Using (\ref{1/mu}) we get
\begin{equation}
-f(0^+)+\frac{1}{\mu(\lambda)}\Delta' f= \left< {\bf h}^1, (\lambda I-[\gamma_i])^{-1} \boldsymbol{\beta} \right>. \label{DC1_eta+}
\end{equation}
If $\lambda=\gamma_I$ then from (\ref{OP1}) we have $\Delta' f= -\frac{ h^1_I}{\beta_I}$. Also for $i\neq I$, $f^1_i=\frac{h^1_i+\beta_i\Delta' f}{\gamma_I-\gamma_i}$.
Thus (\ref{DC1}) gives
\begin{equation}\label{lambda=gammaI}
-f(0^+) -\sigma \frac{h^1_I}{\beta_I}-\sum_{i\neq I} \frac{\beta_i}{\beta_I}\frac{\beta_I h_i^1 - \beta_i h_I^1}{\gamma_I - \gamma_i}  =  \beta_I f^1_I. 
\end{equation}

\noindent $\underline{\kappa > 0}:$ If $\lambda \neq \delta_j$ for all $j$, substituting ${\bf f}^2$ from (\ref{OP2}) into (\ref{DC2}) results in
\begin{equation*}
f'(0^-)-\tau \Delta f - \left< (\lambda I-[\delta_j])^{-1}({\bf h}^2+\boldsymbol{\alpha} \Delta f),\boldsymbol{\alpha}\right> =0,
\end{equation*}
using (\ref{1/nu}) we get
\begin{equation}
f'(0^-)-\frac{1}{\nu(\lambda)}\Delta f= \left< {\bf h}^2, (\lambda I-[\delta_j])^{-1} \boldsymbol{\alpha} \right>. \label{DC2_kappa+}
\end{equation}
If $\lambda=\delta_J$ then from (\ref{OP2}) we have $\Delta f= -\frac{ h^2_J}{\alpha_J}$. Also for $j\neq J$, $f^2_j=\frac{h^2_j+\alpha_j\Delta f}{\delta_J-\delta_j}$.
Thus (\ref{DC2}) gives
\begin{equation}
f'(0^-) +\tau \frac{h^2_J}{\alpha_J}-\sum_{j\neq J} \frac{\alpha_j}{\alpha_J}\frac{\alpha_J h_j^2  - \alpha_j h_J^2}{\delta_J - \delta_j}  =  \alpha_J f^2_J.
\label{lambda=deltaJ}
\end{equation}

\noindent $\underline{\eta= 0} :$ For $\lambda\neq c_i$ for all $i$, combining (\ref{OP3}) with (\ref{DC3}) and using (\ref{mu}) gives
\begin{equation}
\Delta'f-\mu(\lambda)f(0^+)=\left< {\bf h}^1, (\lambda I-[c_i])^{-1}{\bf b}\right>.
\label{DC1_eta0}
\end{equation}
For $\lambda=c_I$ we obtain $f(0^+) = -\frac{h_I^1}{b_I}$ and 
\begin{equation}
\Delta'f +\xi \frac{h^1_I}{b_I}-\sum_{i\neq I} \frac{b_i}{b_I}\frac{b_I h_i^1 - b_i h_I^1}{c_I - c_i} =  b_I f^1_I. \label{lambda=cI}
\end{equation}

\noindent $\underline{\kappa =0} :$ For $\lambda \neq d_j$ for all $j$, combining (\ref{OP4}) with (\ref{DC4}) and using (\ref{nu}) we obtain
\begin{equation}
-\Delta f +\nu(\lambda) f'(0^-)=\left< {\bf h}^2, (\lambda I - [d_j])^{-1}{\bf a}\right>. \label{DC2_kappa0}
\end{equation}
When $\lambda=d_J$ we get $f'(0^-) = -\frac{h_J^2}{a_J}$ and 
\begin{equation}
-\Delta f -\zeta \frac{h^2_J}{a_J}-\sum_{j\neq J} \frac{a_j}{a_J}\frac{a_J h_j^2 - a_j h_J^2}{d_J - d_j} =  a_J f^2_J. \label{lambda=dJ}
\end{equation}

Now using the domain conditions (\ref{DC1_eta+}), (\ref{DC1_eta0}), (\ref{DC2_kappa+}) and (\ref{DC2_kappa0}) we get the following equations for $A$ and $B$:
\begin{align*}
\eta >0: & \quad -\frac{u_-'(0^-)}{\mu(\lambda)} A+\left( \frac{v_+'(0^+)}{\mu(\lambda)} -v_+(0^+) \right)B = 
\left< {\bf h}^1, (\lambda I - [\gamma_i])^{-1}\boldsymbol{\beta}\right>,\\
\eta = 0: & \quad -u_-'(0^-)A + (v_+'(0^+)-\mu(\lambda) v_+(0^+))B=\left< {\bf h}^1, (\lambda I -[c_i])^{-1} {\bf b}\right>,\\ 
\kappa >0: & \quad \left(u_-'(0^-)+ \frac{u_-(0^-)}{\nu(\lambda)} ) \right)A -\frac{v_+(0^+)}{\nu(\lambda)} B = 
\left< {\bf h}^2, (\lambda I- [\delta_j])^{-1}\boldsymbol{\alpha}\right>, \\
\kappa =0: & \quad (u_-(0^-)+\nu(\lambda)u_-'(0^-))A-v_+(0^+)B = \left< {\bf h}^2, (\lambda I - [d_j])^{-1}{\bf a}\right>,
\end{align*}
which can be written in the matrix form
\begin{equation*}
\begin{bmatrix}
-u_-'(0^-) & v_+'(0^+)-\mu(\lambda)v_+(0^+) \\ u_-(0^-)+\nu(\lambda) u_-'(0^-)& -v_+(0^+)
\end{bmatrix} 
\begin{bmatrix} A \\ B \end{bmatrix} =
\begin{bmatrix}
\left<{\bf h}^1, {\bf p} \right> \\\left<{\bf h}^2, {\bf q} \right>
\end{bmatrix},
\end{equation*}
hence giving (\ref{ABcoeff}), where ${\bf p}$, ${\bf q}$ are given by (\ref{p}) and (\ref{q}).

 Note that $D \neq0$ since if $D=0$, then $\left[\begin{array}{c}A\\ B\end{array}\right]$ 
could be taken as a non-zero vector in the null space of the matrix 
$$\begin{bmatrix}
-u_-'(0^-) & v_+'(0^+)-\mu(\lambda)v_+(0^+) \\ u_-(0^-)+\nu(\lambda) u_-'(0^-)& -v_+(0^+)
\end{bmatrix}$$
and the resulting function $f$ would be an eigenfunction of (\ref{SL})-(\ref{TX-1}),
 which contradicts the assumption that $\lambda$ is not an eigenvalue of $L$. Hence $D \ne 0$ and $A$ and $B$ are as given in (\ref{ABcoeff}). 
 
 Thus $f$ has been uniquely determined, provided that $\lambda \neq \gamma_i, c_i$ for all $i$ and $\lambda \neq \delta_j, d_j$ for all $j$. We provide an example at the end to demonstrate how residue calculations are used to determine $A$ and $B$ from (\ref{ABcoeff}) at poles. 
 
 Now $f_i^1$ and $f_j^2$ are uniquely determined by (\ref{OP1}), (\ref{OP3}), (\ref{OP2}) and (\ref{OP4}) for $\lambda \neq \gamma_i$ ($\eta >0$), 
$\lambda \neq c_i$ ($\eta =0$), $\lambda \neq \delta_j$ ($\kappa >0$) and $\lambda \neq d_j$ ($\kappa =0$) respectively.

For $\lambda= \gamma_I$ ($\eta>0$) or $\lambda=c_I$ ($\eta=0$) we have that $f_I^1$ is uniquely determined by (\ref{lambda=gammaI}) and (\ref{lambda=cI}) respectively, while for $\lambda=\delta_J$ ($\kappa>0$) or $\lambda=d_J$ ($\kappa=0$) we have that $f_J^2$ is uniquely determined by (\ref{lambda=deltaJ}) and (\ref{lambda=dJ}) respectively.
\end{proof}

\begin{ex}
We consider the case where $\mu(\lambda)$ has a pole at $\lambda = c_I$ and $\nu(\lambda) \in \mathbb{C}$. For $\lambda \neq c_I$ we can rewrite equation (\ref{ABcoeff}) as
$$ \left[\begin{array}{c} A\\ B\end{array}\right] =\frac{\mu(\lambda)}{D} \left[\begin{array}{cc} -v_+(0^+) & \frac{-v_+'(0^+)}{\mu(\lambda)} + v_+(0^+)\\ -u_-(0^-) - \nu(\lambda) u_-'(0^-) & \frac{-u_-'(0^-)}{\mu(\lambda)}\end{array}\right] \left[ \begin{array}{c} \frac{\left< {\bf h}^1, {\bf p}\right>}{\mu(\lambda)}\\ \left< {\bf h}^2, {\bf q}\right>\end{array}\right].$$
Here, $\frac{D}{\mu(\lambda)} \rightarrow (u_-(0^-) +\nu(\lambda)u_-'(0^-))v_+(0^+)$ as $\lambda \rightarrow c_I$. Note that this limit is non-zero else (\ref{TX-0}) and (\ref{TX-1}) would imply that either $\chi_{[-a,0)}u_-$ or $\chi_{(0,b]} v_+$ are eigenfunctions, contradicting the assumption that $\lambda = c_I$ is not an eigenvalue. Moreover, by (\ref{mu}) we obtain for the case of $\eta =0$ that
$$(\lambda - c_I)\mu(\lambda) = \frac{\xi \prod\limits_{i=1}^N (\lambda - c_i) + \sum\limits_{i=1}^N b_i^2 \prod\limits_{k\neq i} (\lambda - c_k)}{\prod\limits_{i\neq I} (\lambda - c_i)} \rightarrow \frac{1}{b_I^2} \mbox{ as } \lambda \rightarrow c_I.$$
Hence, as $\lambda \rightarrow c_I$,
$$\frac{\left< {\bf h}^1, {\bf p}\right>}{\mu(\lambda)} \rightarrow \begin{cases} \frac{h_I^1}{b_I}, & \eta =0,\\ \left<{\bf h}^1 , (c_I I -[\gamma_i])^{-1} \boldsymbol{\beta}\right>, & \eta >0.\end{cases}$$ 
If, in addition, we have that $\nu(c_I) = 0$ then $c_I = \delta_J$ and from (\ref{1/nu}) we obtain, for the case of $\kappa >0$, that 
$$(\lambda - \delta_J) \frac{1}{\nu(\lambda)} = \frac{\tau \prod\limits_{j=1}^{M'}(\lambda - \delta_j) +\sum\limits_{j=1}^{M'} \alpha_j^2 \prod\limits_{k\neq j} (\lambda - \delta_k)}{\prod\limits_{j\neq J}(\lambda - \delta_j)} \rightarrow \alpha_J^2 \mbox{ as } \lambda \rightarrow \delta_J.$$
So, as $\lambda \rightarrow \delta_J$,
$$\frac{\left< {\bf h}^2, {\bf q}\right>}{\mu(\lambda)} \rightarrow \begin{cases} \left< {\bf h}^2 , (\delta_J I - [d_j])^{-1} {\bf b}\right>, & \kappa = 0\\ \frac{h_J^2}{\alpha_J}, & \kappa >0.\end{cases}$$ 

Thus $A$ and $B$ can be obtained uniquely. In particular, for $\eta = 0$ and $\kappa >0$ with $\lambda = c_I = \delta_J$ we obtain 
 $$A = \frac{1}{u_-(0^-)}\left[ -\frac{h_I^1}{b_I} + \frac{h_J^2}{\alpha_J}\right],\quad  B = -\frac{1}{v_+(0^+)} \frac{h_I^1}{b_I},$$
which agrees with the conditions $f(0^+) = -\frac{h_I^1}{b_I}$ ($\eta = 0$, $\lambda = c_I$) and $\Delta f = -\frac{h_J^2}{\alpha_J}$ ($\kappa>0$, $\lambda = \delta_J$) obtained in the proof of Theorem \ref{resolvent-gamma}.
\end{ex}

\subsection*{Acknowledgment}
The authors would like to thank the referee and editor for their valuable comments.

\end{document}